\documentclass[12pt,leqno]{amsart}
\usepackage{amsmath,amsfonts,amssymb,amsthm}
\usepackage{graphicx}
\graphicspath{ }
\usepackage{wrapfig}
\usepackage{accents}
\usepackage{caption}
\usepackage{subcaption}
\usepackage{calligra}

\numberwithin{figure}{section}

\theoremstyle{plain}
\newtheorem{thm}{Theorem}[section]

\newtheorem{prop}[thm]{Proposition}
\newtheorem{cor}{Corollary}[thm]
\theoremstyle{definition}
\newtheorem{defn}{Definition}[section]

\newtheorem{exmp}{Example}[section]

\theoremstyle{remark}

\title{A note on $p^\lambda$-convex set in a complete Riemannian manifold}
\author[A. A. Shaikh, C. K. Mondal and I. Iqbal]{$^1$Absos Ali Shaikh, $^2$Chandan Kumar Mondal and $^3$Akhlad Iqbal}
\address{\noindent\newline $^{1,2}$Department of Mathematics,\newline University of
Burdwan, Golapbag,\newline Burdwan-713104,\newline West Bengal, India}
\email{aask2003@yahoo.co.in, aashaikh@math.buruniv.ac.in}
\email{chan.alge@gmail.com}
\address{\noindent\newline $^{3}$Department of Mathematics,\newline Aligarh Muslim University,\newline Aligarh-202002,\newline UP, India}
\email{ akhlad6star@gmail.com, akhlad.mm@amu.ac.in}

\begin{document}
\begin{abstract}
In this paper we have generalized the notion of $\lambda$-radial contraction in complete Riemannian manifold and developed the concept of $p^\lambda$-convex function. We have also given a counter example proving the fact that in general $\lambda$-radial contraction of a geodesic is not necessarily a geodesic. We have also deduced some relations between geodesic convex sets and $p^\lambda$-convex sets and showed that under certain conditions they are equivalent.
\end{abstract}
\maketitle
\section{Introduction}
The notion of convexity is a basic topic of modern mathematics, specially, in optimization theory and linear programming. But only convexity is not sufficient to study the behavior of a set. Hence there are many generalization of convexity not only in Euclidean space but also in manifold. Some related work on this topic can be found in \cite{CK17, LAY82, RAP94, UDR94}.\\
\indent In 2010 Beltagy and Shenawy \cite{BS10} defined the notion of $\lambda$-radial contraction in Euclidean space and proved that under such a contraction a line remains invariant. In this paper we have defined $\lambda$-radial contraction in a complete Riemannian manifold and showed that, in general, $\lambda$-radial contraction of a geodesic need not be a geodesic. In fact convexity property of a subset in a Riemannian manifold is not invariant under the $\lambda$-radial contraction and hence a new type of convexity is needed. Motivating by these ideas we have defined $\lambda$-convex set with respect to a point $p$, briefly called $p^\lambda$-convex set, and also $p$-convex set in a complete Riemannian manifold. It is proved that in a complete Riemannian manifold geodesic convexity and $p$-convexity are equivalent under certain conditions. We have also proved that if a set contains an interior point $p$ then there exists some $\lambda$ such that the set is $p^\lambda$-convex. We have also showed that every $p^\lambda$-convex set contains a geodesic convex set.
\section{Radial contraction and $p^\lambda$-convexity}

Let $(M,g)$ be a complete $n$-dimensional Riemannian manifold with Levi-Civita connection $\nabla$. For any two points $x,y\in M$, let $\gamma_{xy}:[0,1]\rightarrow M$ be the length minimizing geodesic \cite{LAN99} from $x$ to $y$ such that $\gamma_{xy}(0)=x$ and $\gamma_{xy}(1)=y$. For a fixed $p$ and $x$ in $M$, consider the set
$$G_p^x=\Big\{\gamma_{px}'(0):\forall \gamma_{px}:[0,1]\rightarrow M, \gamma_{px}(0)=p \text{ and }\gamma_{px}(1)=x\Big\}.$$
Now $G^x_p$ is a subset of $T_p(M)$. Since $(M,g)$ is complete, hence $G_p^x\neq \phi$. Let $\eta_p:M\rightarrow T_p(M)$ be a function such that $\eta_p(x)\in G_p^x\ \forall x\in M$. The function $\eta_p$ is called direction function at $p$.
\begin{defn}\cite{RAP94}
A subset $A$ of $M$ is called \textit{geodesic convex} if for any two points $x,y\in A$ there exists a geodesic $\gamma_{xy}:[0,1]\rightarrow M$ such that $\gamma_{xy}(0)=x$ and $\gamma_{xy}(1)=y$ and $\gamma_{xy}(t)\in A, \forall t\in [0,1]$
\end{defn} 
\begin{defn}
For a fixed point $p\in M$ and for fixed $\lambda\in (0,1]$ define $\eta_pC^\lambda:M\rightarrow M$ by
$$\eta_pC^\lambda(x)=\gamma_{px}(\lambda)\quad \text{for }x\in M,$$
where $\gamma_{px}:[0,1]\rightarrow M$ is a geodesic such that $\gamma_{px}'(0)=\eta_p(x)$.
This function is called $\lambda$-\emph{radial contraction} of $x$ based at $p$ with respect to $\eta_p$.
\end{defn}
Since there is exactly one length minimizing geodesic between any two points $p$ and $x$ in $(M,g)$ whose initial vector is $\eta_p(x)$ hence $\eta_pC^\lambda$-function is well defined in $(M,g)$.\\
\indent Beltagy and Shenawy \cite{BS10} studied such type of function in an Euclidean space. They called it $\lambda$-\textit{radial contraction} based at $p$.
\begin{defn}\cite{BS10}
Let $A$ be a nonempty subset of $\mathbb{R}^n$. For a fixed point $p\in A$ and a fixed real number $\alpha\in(0,1)$, The $\lambda$-\textit{radial contraction} of $A$ based at
$p$ is denoted by $C^\lambda_p(A)$ and is defined by
$$C^\lambda_p(A)=\{\lambda x+(1-\lambda)p: x\in A\}.$$

\end{defn}
 We have generalized this notion in a complete Riemannian manifold and developed a new type of convex set in $(M,g)$.\\
\indent For a fixed $\lambda\in [0,1]$, the $\lambda$-radial contraction of a non-empty subset $A$ based at $p$ with respect to $\eta_p$ can be defined as
$$\eta_pC^\lambda(A)=\Big\{\eta_pC^\lambda(x):x\in A\Big\}.$$
In $\mathbb{R}^n$, for each $x\in \mathbb{R}^n$,  $\eta_p(x)$ has only one choice hence we denote $\eta_pC^\lambda$ by $C_p^\lambda$. \\
\indent In \cite{BS10} it has been shown that $\lambda$-radial contraction of a line segment is also a line segment. But in general this does not hold, see Example below.
\begin{exmp}
Take $M=\mathbb{S}^2$ with the line element $ds^2=d\theta^2+\sin^2\theta d\phi^2$. Now take two points $x,y$ on the equator and take $p$ as the north pole of $\mathbb{S}^2$. Then if we choose $\lambda=\frac{1}{2}$, then $\alpha=\eta_pC^{\frac{1}{2}}(\gamma_{xy})$ is not a geodesic, see Figure \ref{sphere1}, since the circle containing $\alpha$ has radius less than one, hence it can not be a geodesic.
\end{exmp}
 \begin{figure}[h]
     \centering
     \includegraphics[width=.40\textwidth]{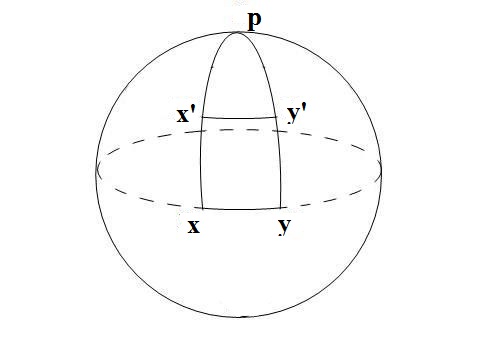} 
     \caption{} 
        
     \label{sphere1}
 \end{figure}
 Naturally the following question arises.

 What is the sufficient condition that the map $\eta_pC^\lambda$ transforms geodesic to geodesic for a fixed $\eta_p$?

\begin{defn}
Let $p\in M$ and $A\neq \phi$ be a subset of $M$ with the direction function $\eta_p$. Then $A$ is called $p^\lambda$-\textit{convex} with respect to $\eta_p$ for a fixed $\lambda\in (0,1]$ if for any two points $x,y\in A$ there is a length minimizing geodesic connecting $\eta_pC^\lambda(x)$ and $\eta_pC^\lambda(y)$ belongs to $A$, i.e, $\gamma_{x'y'}(t)\in A, \forall t\in [0,1]$, where $x'=\eta_pC^\lambda(x)$ and $y'=\eta_pC^\lambda(y)$. 
\end{defn}
Unless there are some confusion, by a $p^\lambda$-convex set we always mean $p^\lambda$-convex with respect to some direction function $\eta_p$. 
\begin{defn}
If a non-empty subset $A$ of $M$ is $p^\lambda$-convex for all $\lambda \in (0,1]$, then $A$ is called \emph{totally $p$-convex}.
\end{defn}
\begin{thm}
Let $A$ be a subset of $M$ containing more than two points. Then $A$ is geodesically convex set in $(M,g)$ if and only if for each point $p\in A$ there exists a  direction function $\eta_p$ such that $A$ is $p^\lambda$-convex with respect to $\eta_p\ \forall \lambda\in (0,1]$, i.e, $A$ is totally $p^\lambda$-convex. 
\end{thm}
\begin{proof}
Let $A$ be a geodesically convex subset of $M$. Take a point $p\in A$. By convexity for any point $x\in A$ there exists a geodesic $\gamma_{px}:[0,1]\rightarrow M$ such that $\gamma_{px}(0)=p$, $\gamma_{px}(1)=x$ and $\gamma_{px}(t)\in A\ \forall t\in[0,1]$. Define the function $\eta_p:M\rightarrow T_p(M)$ by
$$\eta_p(x)=\gamma_{px}'(0)\quad \forall x\in A.$$
Now for any two points $x,y$ in $A$, the points $\eta_pC^\lambda(x), \eta_pC^\lambda(y)\in A\ \forall \lambda\in (0,1]$. Hence there exists a unique geodesic connecting $\eta_pC^\lambda(x)$ and $\eta_pC^\lambda(y)$ that will belongs to $A$ for each $\lambda\in (0,1]$. Hence $A$ is totally $p^\lambda$-convex with respect to $\eta_p$.\\
\indent Conversely let $A$ be totally $p$-convex for each $p\in A$. Take any two points $x,y$ in $A$. But $x=\eta_pC^\lambda(x)$ and $y=\eta_pC^\lambda(y)$, where $\lambda=1$. So using the definition of totally $p$-convex we can say that there is a geodesic arc connecting $x$ and $y$ belongs to $A$. So $A$ is geodesically convex. 
\end{proof}
\begin{cor}
For a subset $A$ in $M$ containing more than two points the following statements are equivalent:\begin{enumerate}
\item $A$ is geodesically convex,
\item $A$ is totally $p$-convex for any $p\in A$.
\end{enumerate} 
\end{cor}
\begin{prop}
In $\mathbb{R}^n$, a subset $A$ is $p^\lambda$-convex for some $p\in A$ and $\lambda\in (0,1]$ if and only if $C_p^\lambda(A)$ is $p^{\lambda_1}$-convex for some $\lambda_1\in (0,1]$.
\end{prop} 
\begin{proof}
The proof can be easily deduced from the fact that in $\mathbb{R}^n$, the $\lambda$-contraction of a straight line is also a straight line \cite{BS10}.
\end{proof}
But in case of manifold the above proposition is not true, see example \ref{ex2}.
\begin{exmp}\label{ex2}
Take $M=\mathbb{S}^2$ and for two fixed points $x,y$ on the equator let $A$ be the set $\{\gamma_{xy}(t):t\in [0,1]\}\cup \{p\}$, where $p$ be the north pole, see figure \ref{sphere1}. Then $A$ is $p$-convex. Now for some fixed $\eta_p$ the $p^{1/2}$-radial contraction of $A$ is the set $\{\text{the shortest latitude joining } \eta_pC^{\frac{1}{2}}(x)$ and $\eta_pC^\frac{1}{2}(y)\}\cup \{p\}$ but this set is not $p^\lambda$-convex for any $\lambda\in (0,1]$.
\end{exmp}
\indent One of the main difference between geodesically convex set and $p^\lambda$-convex set is that geodesically convex sets are always path-connected, more specifically geodesically connected but $p^\lambda$-convex set may be disconnected also, e.g. consider $M=\mathbb{S}^2$ and take upper hemisphere together with two distinct points in lower hemisphere, then this set is not path connected but it is $p^\lambda$-convex set for some $\lambda\in(0,1]$, where $p$ is the noth pole of $\mathbb{S}^2$. But not all disconnected sets are $p^\lambda$-convex, e.g.- any finite set containing more than two points in $\mathbb{R}^n$ is not $p^\lambda$-convex for some $\lambda\in(0,1]$.
\begin{prop}\label{open1}
If a subset $A\subset M$ contains an interior point $p$ then $\exists \ \zeta\in (0,1]$ such that $A$ is $p^\lambda$-convex $\forall \lambda\in (0,\zeta]$.
\end{prop}
\begin{proof}
Let $p\in A$ be an interior point. Then $B(p,r)\subset A$ for some $r>0$. Take $V_p=B(p,r)\cap N_p$, where $B_p$ is the geodesic ball $p$. Now $V_p$ looks like a flat $n$-dimensional Euclidean space. For each $x$ in $A$ take $\lambda_x=\inf\{\lambda\in (0,1]: \eta_pC^\lambda(x)\in V_p\}$. $\lambda_x\neq 0$ since $V_p$ is open so sufficiently small portion of any geodesic emitting from $p$ must lies in $V_p$. Again take $\zeta=\inf\{\lambda_x:x\in A\}$. Hence $A$ be $p^\lambda$-convex for all $\lambda\in (0,\zeta]$.
\end{proof}
\begin{prop}\label{1}
Let $A$ be a nonempty subset of $M$ and $\lambda,\beta\in (0,1]$. Then for $p\in M$ 
$$C^\lambda_p(C^\beta_p(x))=C^{\lambda\beta}_p(x) \text{ for }x\in A.$$
\end{prop}
 \begin{prop}
 Let $A\subset M$ be $p^\lambda$-convex set. Then $A$ is also $p^{\lambda^n}$-convex $\forall n\in \mathbb{N}$.
 \end{prop}
 \begin{proof}
 Fixed $m\in \mathbb{N}$ and choose two points $x,y$ from $A$. Since $A$ is $p^\lambda$-convex, hence $\gamma_{x_1y_1}$ belongs to $A$ where $x_1=C^\lambda_p(x)$ and $y_1=C^\lambda_p(y)$. Again by similar argument taking $x_2=C^\lambda_p(x_1)$ and $y_2=C^\lambda_p(y_1)$ we get $\gamma_{x_2y_2}$ belongs to $A$. Hence continuing this way we get $\gamma_{x_ny_n}\in A$ where $x_n=C^\lambda_p(x_{n-1})$ and $y_n=C^\lambda_p(y_{n-1})$. Then from Proposition \ref{1} we get $x_m=C^{\lambda}_p\circ C^{\lambda}_p\circ \cdots \circ C^{\lambda}_p(x)$ (m-times) and $y_m=C^{\lambda}_p\circ C^{\lambda}_p\circ \cdots \circ C^{\lambda}_p(y)$ (m-times). So, $C^{\lambda^m}_p(x)$, $C^{\lambda^m}_p(y)$ and the geodesic arc joining these two points belong to $A$. Hence $A$ is $p^{\lambda^m}$-convex for any $m\in \mathbb{N}$.
 \end{proof}
 \begin{thm}
 Let $\{A_i\}_{i\in \lambda}$ be an arbitrary collection of $p^\lambda$-convex sets and $\bigcap_{i\in\Lambda}A_i$ is nonempty. Then $\bigcap_{i\in\Lambda}A_i $ is also $p^\lambda$-convex.
 \end{thm}
 \begin{proof}
 Let $x,y\in \bigcap_{i\in\Lambda}A_i$. Hence $x,y\in A_i$ for all $i\in \Lambda$. Now by definition of $p^\lambda$-convex set, we have $\eta _pC^\lambda(x), \eta _pC^\lambda(y)\in A_i \ \forall i\in\Lambda$. Hence $\bigcap_{i\in\Lambda}A_i$ is $p^\lambda$-convex.
 \end{proof}
 \begin{thm}
 For any $p^\lambda$-convex set $A\subset M$ there exists a geodesically convex subset $V$ containing $p$ such that $V\subset A$. 
 \end{thm}
 \begin{proof}
 Let $A$ be $p^\lambda$-convex set. Take $$V=\bigcup_{x,y\in A} \Big\{\gamma_{x'y'}:x'=C^\lambda_p(x),y'=C^\lambda_p(y)\Big\}.$$ We shall show that $V$ is geodesically convex. Choose any two points $x,y$ from $B$. Then $x=C^\lambda_p(x')$ and $y=C^\lambda_p(y')$ for some $x',y'\in A$. Since $A$ is $p^\lambda$-convex, the geodesic $\gamma_{xy}\in A$ implies $\gamma_{xy}$ belongs to $V$. Hence $V$ is geodesically convex.
 \end{proof}
 \section*{Acknowledgement}
 The second author is thankful to The University Grants Commission, Government of India for giving the award of Junior Research Fellowship
 during the tenure of preparation of this research paper.

\end{document}